\documentclass[12pt,twoside]{amsart}  
\usepackage{amssymb,verbatim,amsmath}  
\usepackage{color}
 \usepackage{hyperref}
  
\numberwithin{equation}{section}  
\newtheorem{theorem}{Theorem}[section]  
\newtheorem{lemma}[theorem]{Lemma}  
\newtheorem{proposition}[theorem]{Proposition}  
\newtheorem{corollary}[theorem]{Corollary}

\theoremstyle{definition}  
\newtheorem{definition}[theorem]{Definition}
\newtheorem{remark}[theorem]{Remark}  
\newtheorem{example}[theorem]{Example}

\newtheorem{notation}[theorem]{Notation}  
\newtheorem{question}[theorem]{Question}  

\DeclareMathOperator{\ass}{ass}

\DeclareMathOperator{\Ann}{Ann}  
\DeclareMathOperator{\Ass}{Ass}

\usepackage{tikz}
\usetikzlibrary{positioning}

\begin{document}  
    
%%%%%%%%%%%%%%%%%%%%%%%%%%%%%%%%%%%%%%%%%%%%%%%%%%%%%%%%%%%%%%%%%%%%%%%%%%%%%%%%%%%  
   
\title[In the Shadows of a hypergraph]{In the Shadows of a hypergraph: looking for associated primes of powers of squarefree monomial ideals}
  
\author{Erin Bela}  
\address{Department of Mathematics, 255 Hurley, University of Notre Dame, Notre Dame, IN 46556}  
\email{ebela@nd.edu}  
\urladdr{}

\author{Giuseppe Favacchio}  
\address{Dipartimento di Matematica e Informatica, Viale A. Doria, 6 - 95100 - Catania, Italy}  
\email{favacchio@dmi.unict.it}  
\urladdr{http://www.dmi.unict.it/~gfavacchio}

\author{Nghia Tran}  
\address{School of Mathematics, Statistics and Applied Mathematics, National University of Ireland, Galway, Ireland}  
\email{n.tranthihieu1@nuigalway.ie}  
\urladdr{}  
  
\thanks{Version: \today}

\maketitle  

%%%%%%%%%%%%%%%%%%%%%%%%%%%%%%%%%%%%%%%%%%%%%%%%%%%%%%%%%%%%%%%
\begin{abstract}
	The aim of this paper is to study the associated primes of  powers of squarefree monomial ideals. Hypergraphs and squarefree monomial ideals are strongly connected. 	The cover ideal $J(H)$ of a hypergraph $H$ is the intersection of the primes corresponding to the edges of $H$.  We define the $shadow$ of $H$, as a certain set of smaller hypergraphs related to $H$. We then describe how the shadows of $H$ preserve information about the associated primes of the powers of $J(H)$. Some implications to the persistence property are studied.	
\end{abstract}
  
%%%%%%%%%%%%%%%%%%%%%%%%%%%%%%%%%%%%%%%%%%%%%%%%%%%%%%%%%%%%%%%%%%% 
  
\section*{Introduction} \label{s.intro}  
The primary decomposition of ideals in Noetherian rings is a fundamental result in commutative algebra and algebraic geometry. 
From a minimal primary decomposition, one can define the set of the associated primes by taking the radical of each ideal in the decomposition. Our goal in this paper is to investigate the associated primes of powers of squarefree monomial ideals.

Squarefree monomial ideals and powers of ideals are central objects in combinatorial and commutative algebra and algebraic geometry for the several connections they encode between these areas, see for instance \cite{MS}. 

There are several ways to relate a squarefree monomial ideal to a hypergraph. To serve our intent, we will associate to a hypergraph the squarefree monomial ideal with minimal primes corresponding to the edges of the hypergraph, and vice versa.  This ideal is usually called the cover ideal of the hypergraph. 

The associated primes of a squarefree monomial ideal are easy to describe, whereas computing the associated primes of a power could be really tricky. Currently, the set of the associated primes of a power of any squarefree monomial ideal is far from being fully understood. 
Recently, in \cite{TMR} Kaiser, Stehl\' ik and \u Skrekovski produced an example of a squarefree monomial ideal, precisely the cover ideal of a graph, which fails the persistence property, i.e., the set of the associated primes could ``lose" some elements from a power to the next. 

For ideals associated to a combinatorial object, one hopes to explain their behavior in terms of the original object.  With this in mind, we define the shadow of a hypergraph, Definition \ref{d.shadow}, as a certain set of smaller hypergraphs  related to the original one. We show that, see Theorem \ref{t.oddcycle}, Theorem \ref{t.tildeH}, the shadows preserve information about the associated primes of a power of the cover ideal of the hypergraph.  

The paper is organized as follows. 
In Section \ref{s.prelims}, we introduce the terminology and the basic results. 
In Section \ref{s.intro shadow}, we define the shadows of a hypergraph that are the new tool introduced in this paper. Then we start an investigation of the associated primes of a squarefree monomial ideal in terms of the shadows of the associated hypergraph. In particular, in this section, we deal with the second power.
In Section \ref{s.main case}, under some restrictive conditions, we broaden our investigation to any power.
Finally, in Section \ref{s.application}, we apply the results of Section \ref{s.main case} to the persistence property. 

\textbf{Acknowledgments.} 
This project started during the PRAGMATIC 2017 Research
School ``Powers of ideals and ideals of powers" held in Catania, Italy. We would like to thank the University of Catania and the organizers of the workshop, Alfio Ragusa, Elena Guardo, Francesco Russo, and Giuseppe Zappal\`a.  As well, we are deeply grateful to Adam Van Tuyl for introducing us to this topic and for his guidance and to Huy T\`ai H\`a for the useful comments. We also thank them together with Brian Harbourne  and Enrico Carlini for their inspiring lectures. Computations were carried out with CoCoA \cite{COCOA} and Macaulay2 \cite{MACAULAY2}.

%%%%%%%%%%%%%%%%%%%%%%%%%%%%%%%%%%%%%%%%%%%%%%%%%%%%%%%%%%%%%%%%%%% 
 
\section{Notation and basic facts} \label{s.prelims}
   Let $V := \{x_1,\ldots,x_n\}$ and $R=K[V]=K[x_1,\ldots, x_n]$ be the standard polynomial ring in $n$ variables over a field $K$.
   A squarefree monomial ideal $I\subseteq R$ always has a unique minimal primary decomposition, $I=\mathfrak p_1\cap \cdots \cap \mathfrak p_t$,  as an intersection  of squarefree prime ideals $\mathfrak p_i=(x_{i_1}, \ldots, x_{i_s})$. For more details and a full description of the topic we refer to Section 1.3 in \cite{HH}.
  
   This property establishes a one-to-one correspondence between squarefree monomial ideals and finite simple hypergraphs.
   First, recall that a (simple) {\bf hypergraph} $H$ is a pair $H=(V,E)$ where $V := \{x_1,\ldots,x_n\}$ is called the set of vertices of $H$ and $E$ is a collection of subsets of $V$.  (We will only consider finite simple hypergraphs, these are also called $clutters$ in the literature).
   We will denote, for a set $U=\{x_{i_1},\ldots,x_{i_s}  \}\subseteq V$,  by $$\mathfrak p_{U}:= (x_{i_1},\ldots,x_{i_s}) \subseteq R$$ the prime ideal generated by the variables in $U$ and by 
   $$x_U:=x_{i_1}\cdots x_{i_s} \in R$$ 
   the monomial given by the product of the variables in U.
   
   Then, a hypergraph $H=(V,E)$ unequivocally corresponds to the squarefree monomial ideal $J(H):=\bigcap\limits_{e\in E} \mathfrak p_e$, called the {\bf cover ideal} of $H$, and vice versa.
   
   Let $H = (V,E)$ be a hypergraph. A subset $T$ of $V$ is a {\bf vertex cover} of $H$ if every edge $e\in E$ contains at least one element of $T$.  A vertex cover $T$ is a {\bf minimal vertex cover} if no proper subset of $T$ is a vertex cover.  
   Minimal vertex covers are related to the minimal generators of $J(H)$. Indeed, $T$ is a minimal vertex cover of $H$ if and only if $x_T\in \mathcal G(J(H)),$  the set of monomials which minimally generates $J(H).$
   See \cite{HVT} and \cite{HVT2} for a further investigation on cover ideals of hypergraphs. 
         
   In this paper we are interested in the study of the associated prime ideals of the (regular) powers of $J(H)$. Recall the following, classical,  definition. 
   \begin{definition} \label{d.assocprime} Let $R$ be a ring and $I$ an ideal of $R$. A prime ideal $\mathfrak p \subset R$ is called an 
 	{\bf associated prime ideal} of $I$ if there exists some element $m \in R/I$ such that $\mathfrak p=\Ann(m)$, the annihilator of $m$. The set of all associated prime ideals of $I$ is denoted by $\ass(I)$.
   \end{definition}  
   
    By definition, the hypergraph $H=(V,E)$ easily provides a description of all elements in  $\ass(J(H))$. Indeed,  $\mathfrak p_U\in  \ass (J(H))$ if and only if $U\in E$. 
    
    In order to describe the associated primes of the powers of $J(H)$ the next lemma is an essential tool. Recall that for an hypergraph $H = (V,E)$ and $U\subseteq V$ the {\bf induced subhypergraph} of $H$ on $U$ is the hypergraph  $H_U=(U,E(U))$ where $E(U)=\{e\in E\ |\ e\subseteq U \}$.  
 
     \begin{lemma} [Lemma 2.11 \cite{FHVT}]\label{l.localization} Let $H=(V,E)$ be a hypergraph.	Let $U\subseteq V$, then
 	\[
 	\mathfrak p_U \in \Ass(R/J(H)^s) \Leftrightarrow \mathfrak p_U \in \Ass(K[U]/J(H_U)^s).
 	\]	
    \end{lemma}
    
    Lemma \ref{l.localization} moves the problem of looking for associated prime ideals to maximal ideals. Indeed,  $\mathfrak p_U$ is associated to  $J(H)^s$, if and only if it is associated to $J(H_U)^s\subseteq K[U]$, and $\mathfrak p_U$ is the maximal ideal in $K[U]$. 
    An immediate consequence of Lemma \ref{l.localization} is a first, well known, step in the description of the elements in $\ass J(H)^s.$
    \begin{lemma}\label{l.edges}
    	Let $H$ be a hypergraph. Then $\mathfrak p_e\in \ass J(H)^s$ for each integer $s\ge 1$ and for each edge $e$ of $H$.
    \end{lemma}
   
   A corollary of Lemma \ref{l.localization} will be useful in Section \ref{s.main case}. 
   \begin{corollary}\label{c.prime of subhypergraph} Let $H=(V,E)$ be a hypergraph on $V$.
   	Let $F\subseteq U\subseteq V$, then  
   	\[
     \mathfrak p_F \in \Ass(R/J(H)^s) \Leftrightarrow \mathfrak p_F \in \Ass(K[U]/J(H_U)^s).
   	\]
   \end{corollary} 
   \begin{proof}
   	It is an immediate consequence of Lemma \ref{l.localization}. Indeed, since $F\subseteq U$ we get $(H_U)_F=H_F.$
   \end{proof}

    In the literature there are only few other results explicitly describing the elements in $\ass (J(H)^s).$ Most of them deal with the case that $H$ is a graph, i.e., the edges all have cardinality 2. If $H$ is a graph we will often denote it by the letter $G$.
    For instance, see proposition below, the authors of \cite{FHVT2} describe the set $\Ass(R/J(G)^2)$. They prove that the new primes match the (minimal) odd cycles of $G.$   
    
    Recall that in a graph $G=(V,E)$ a set of distinct vertices $C =\{ x_{i_1}, x_{i_2}, \ldots, x_{i_n}\}\subseteq V$  is called  an\textit{ $n$-cycle} (or cycle of length $n$)  if $\{x_{i_j}, x_{i_{j+1}}\}\in E$ for each  $j \in \{1,\ldots,n\}$ and $x_{i_{n+1}} := x_{i_1}$.   
 	$C$ is called an odd (even) cycle if $n$ is odd (even). 
 	The vertices $x_{i_j}$, $x_{i_{j+1}}$ connected by an edge $\{x_{i_j}, x_{i_{j+1}}\}$ are called {\it adjacent} vertices in $C$. A {\it chord} of $C$ is an edge of $G$ joining two nonadjacent vertices. If $C$ has no chord, we shall call it {\it chordless}.

 	\begin{proposition}[Corollary 3.4, \cite{FHVT2}]\label{Corollary 3.5 cite{FHVT2}}
	 	Let $G$ be a finite graph.
	 	A prime ideal $\mathfrak p=(x_{i_1},\dots,x_{i_s})$ is in $\ass(J(G)^2)$ if and only if:  
	 	\begin{itemize}  
	 		\item[(a)] $s=2$ and $\mathfrak p\in \ass(J(G))$; or 
	 		\item[(b)] $s$ is odd, and after re-indexing, $\{x_{i_1}, x_{i_2}, \ldots, x_{i_s}\}$ is a chordless cycle of $G$.  
	 	\end{itemize}  
    \end{proposition}

\section{Introducing the shadows}\label{s.intro shadow}
 
	 The authors of \cite{FHVT} give a description of the set $\ass(J(H)^s)$ in terms of the coloring properties of the hypergraph $H$. However, their method is not very efficient to list all the elements in  $\ass(J(H)^s)$ for any given hypergraph $H$. In this section, we introduce a tool that can be useful for this aim: we define the shadows of a hypergraph. The motivating idea is to take information from some other hypergraphs, smaller than $H$,  and to bring it to $H$.
	
	The following is the definition of a shadow of $H.$
	\begin{definition}\label{d.shadow}
		Let $H=(V,E)$ be a hypergraph. We say that a hypergraph  $H'=(V', E')$ is a \textit{shadow} of $H$ if 
     \begin{itemize}
		\item[(a)] $V'\subseteq V$; and
	 	\item[(b)] $|E|=|E'|$ (same cardinalities) and  $e\cap V' \in E'$ for each $e\in E.$ 
	 \end{itemize}  
	\end{definition}
	We denote by $\mathcal{S}(H)$ the set of all the shadows of $H.$  Note that two different elements in $\mathcal{S}(H)$ have different vertex sets.  Thus $H'=(V',E')\in \mathcal{S}(H)$ will be also called the shadow of $H$ on $V'$. By definition,
	$H$ is always a shadow of itself on the vertex set $V$; we refer to this as the \textit{trivial} shadow. However, not every subset of $V$ produces a shadow of $H$, as we show in the following example. 

	\begin{example}\label{e.triangles} Consider the hypergraph $H$ on the vertex set $V= \{x_1,x_2,x_3,x_4,x_5\}$ with the edge set $E=\left\{  \{x_1,x_2,x_3\}, \{x_2,x_3,x_4\},\{x_1,x_4,x_5\}  \right\}$. 
	The set $\mathcal S(H)$ contains non-trivial elements, namely, shadows on the vertex sets $V_1:=\{x_1,x_2,x_4\}$, $V_2:=\{x_1,x_3,x_4\},$ $V_3:=\{x_1,x_2, x_3,x_4\}$, $V_4:=\{x_1,x_3,x_4,x_5\}$ and $V_5:=\{x_1,x_2,x_4,x_5\}.$ Indeed, we have
	\begin{align*}
		\left(V_1, \{ \{x_1,x_2\}, \{x_2,x_4\},\{x_1,x_4\} \}\right)&\in \mathcal S(H), \text{ and}\\
		\left( V_2,\{ \{x_1,x_3\}, \{x_3,x_4\},\{x_1,x_4\} \}\right)&\in \mathcal S(H).
	\end{align*}
	Both of these shadows are graphs, more precisely they are 3-cycles. Additionally, 
	\begin{align*}
		&\left( V_3, \{ \{x_1,x_2, x_3\}, \{x_2, x_3,x_4\},\{x_1,x_4\} \} \right)\in \mathcal S(H), \\
		&\left( V_4, \{ \{x_1,x_3\}, \{x_3,x_4\},\{x_1,x_4,x_5\} \} \right)\in \mathcal S(H) \text{ and }\\
		& \left( V_5, \{ \{x_1,x_2\}, \{x_2,x_4\},\{x_1,x_4,x_5\} \} \right)\in \mathcal S(H).
	\end{align*}
	Furthermore, %we can verify that these are the only subsets which give shadows of $H$. F
	for instance, $H$ has no shadow on the set $V_6:=\{x_1,x_2,x_3\}$ since we get   
	$$\left( V_6, \{ \{x_1,x_2,x_3\}, \{x_2,x_3\},\{x_1\} \}\right),$$
	and this fails to be a simple hypergraph. 
	\end{example}
	The hypergraph $H$ and its shadows are showed in the following figures
\begin{center}
\begin{tikzpicture}[scale=0.7]
    \node (v1) at (0,0) {};
    \node (v2) at (2.5,0) {};
    \node (v3) at (1.25,2.17) {};
    \node (m1) at (1.88,1.08)[right]{$x_3$};
    \node (m2) at (1.25,0) [below]{$x_3$};
    \node (m3) at (0.63,1.08) [left]{$x_5$};
    \node (f1) at (1.25,-1) [below]{Hypergraph $H$};
    \draw [thick] (v1)--(v2);
    \draw [thick] (v2)--(v3);
    \draw [thick] (v3)--(v1);
    \fill (v1) circle (0.1) node [below left] {$x_4$};
    \fill (v2) circle (0.1) node [below right] {$x_2$};
    \fill (v3) circle (0.1) node [above] {$x_1$};
    \node (v4) at (6,0) {};
    \node (v5) at (8.5,0) {};
    \node (v6) at (7.25,2.17) {};
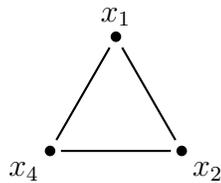
    \node (f2) at (7.25,-1) [below]{Shadow on $V_1$};
    \draw [thick] (v4)--(v5);
    \draw [thick] (v5)--(v6);
    \draw [thick] (v6)--(v4);
    \fill (v4) circle (0.1) node [below left] {$x_4$};
    \fill (v5) circle (0.1) node [below right] {$x_2$};
    \fill (v6) circle (0.1) node [above] {$x_1$};
     \node (a1) at (12,0) {};
    \node (a2) at (14.5,0) {};
    \node (a3) at (13.25,2.17) {};
    \node (m1) at (13.88,1.08)[right]{$x_3$};
    \node (m2) at (13.25,0) [below]{$x_3$};
%    \node (m3) at (0.63,1.08) [left]{$x_5$};
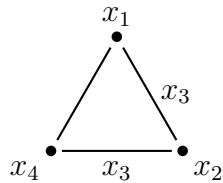
    \node (f4) at (13.25,-1) [below]{Shadow on $V_3$};
    \draw [thick] (a1)--(a2);
    \draw [thick] (a2)--(a3);
    \draw [thick] (a3)--(a1);
    \fill (a1) circle (0.1) node [below left] {$x_4$};
    \fill (a2) circle (0.1) node [below right] {$x_2$};
    \fill (a3) circle (0.1) node [above] {$x_1$};
     \node (v7) at (18,0) {};
    \node (v8) at (20.5,0) {};
    \node (v9) at (19.25,2.17) {};
    \node (m3) at (18.63,1.08) [left]{$x_5$};
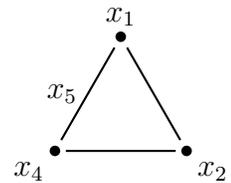
    \node (f3) at (19.25,-1) [below]{Shadow on $V_5$};
    \draw [thick] (v7)--(v8);
    \draw [thick] (v8)--(v9);
    \draw [thick] (v9)--(v7);
    \fill (v7) circle (0.1) node [below left] {$x_4$};
    \fill (v8) circle (0.1) node [below right] {$x_2$};
    \fill (v9) circle (0.1) node [above] {$x_1$};
\end{tikzpicture}
\end{center}
where the edge $\{a,b,v_1,\ldots, v_m\}$ is depicted as the segment 
\begin{tikzpicture}[scale=0.7]
    \node (v1) at (0,0) {};
    \node (v2) at (2.5,0) {};
    \node (m2) at (1.25,0) [above]{$v_1 \ldots v_m$};
    \fill (v1) circle (0.1) node [left] {$a$};
    \fill (v2) circle (0.1) node [right] {$b$};
 \draw [thick] (v1)--(v2);
\end{tikzpicture}. Similarly, we can re-picture the hypergraph $H$ in the following form to better see the shadows on $V_2$ and $V_4$: 
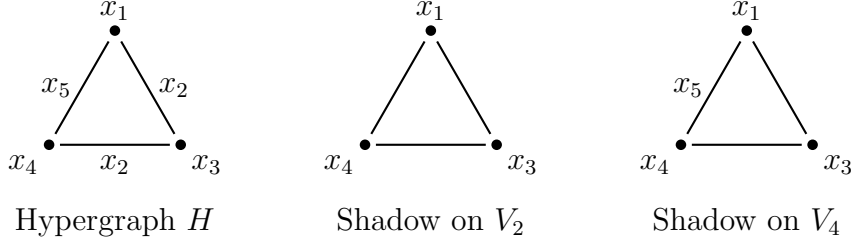
\begin{center}
\begin{tikzpicture}[scale=0.7]
    \node (v1) at (0,0) {};
    \node (v2) at (2.5,0) {};
    \node (v3) at (1.25,2.17) {};
    \node (m1) at (1.88,1.08)[right]{$x_2$};
    \node (m2) at (1.25,0) [below]{$x_2$};
    \node (m3) at (0.63,1.08) [left]{$x_5$};
    \node (f1) at (1.25,-1) [below]{Hypergraph $H$};
    \draw [thick] (v1)--(v2);
    \draw [thick] (v2)--(v3);
    \draw [thick] (v3)--(v1);
    \fill (v1) circle (0.1) node [below left] {$x_4$};
    \fill (v2) circle (0.1) node [below right] {$x_3$};
    \fill (v3) circle (0.1) node [above] {$x_1$};
    \node (v4) at (6,0) {};
    \node (v5) at (8.5,0) {};
    \node (v6) at (7.25,2.17) {};
    \node (f2) at (7.25,-1) [below]{Shadow on $V_2$};
    \draw [thick] (v4)--(v5);
    \draw [thick] (v5)--(v6);
    \draw [thick] (v6)--(v4);
    \fill (v4) circle (0.1) node [below left] {$x_4$};
    \fill (v5) circle (0.1) node [below right] {$x_3$};
    \fill (v6) circle (0.1) node [above] {$x_1$};
     \node (v7) at (12,0) {};
    \node (v8) at (14.5,0) {};
    \node (v9) at (13.25,2.17) {};
    \node (m3) at (12.63,1.08) [left]{$x_5$};
    \node (f3) at (13.25,-1) [below]{Shadow on $V_4$};
    \draw [thick] (v7)--(v8);
    \draw [thick] (v8)--(v9);
    \draw [thick] (v9)--(v7);
    \fill (v7) circle (0.1) node [below left] {$x_4$};
    \fill (v8) circle (0.1) node [below right] {$x_3$};
    \fill (v9) circle (0.1) node [above] {$x_1$};
\end{tikzpicture}
\end{center}
	In the following example, we show a hypergraph which only has trivial shadow.
\begin{example}
	Let $H$ be the hypergraph on the vertex set $V= \{x_1,x_2,x_3,x_4,x_5\}$ and the edge set $E=\left\{  \{x_1,x_2,x_3\}, \{x_2,x_3,x_4\},\{x_3,x_4,x_5\},\{x_4,x_5,x_1\},\{x_5,x_1,x_2\}  \right\}$. In this case, the set $\mathcal{S}(H)$ has only one element, namely $H$. Indeed, notice that each edge of $H$ contains vertices with ``consecutive" indexes. Since any subset of $V$ with two elements is contained in some edge, then $H$ has no shadow on any set $V'\subsetneq V$. For instance, $H$ has no shadow on the subset $V'$  obtained from $V$ by removing  $x_1$ since  $\{x_4,x_5\} \subset \{x_3,x_4,x_5\}$.
\end{example}

	$J(H')$ is an ideal of $K[V']$ and there is a natural inclusion from $K[V']$ into $K[V]$. The ideal generated by the image of $J(H')$ under this map, i.e., the ideal generated by $\mathcal G(J(H'))\subseteq K[V]$, is called cone ideal of $J(H')$ in $K[V].$ 

The next lemma provides a connection between the monomial generators of $J(H)$ and $J(H')$ for a shadow $H'$ of $H$.

\begin{lemma} \label{l.gens Shad are gens Hyper}
	Let $H=(V,E)$ be a hypergraph and $H'=(V',E')\in \mathcal{S}(H)$ a shadow of $H$. Then $\mathcal G(J(H'))\subseteq \mathcal G(J(H))$.
\end{lemma}
\begin{proof} 
	The ideal $J(H')$ is generated by monomials $x_U$ where $U$ is a minimal vertex cover of $H'$.  By the definition of shadow, $U$ is also a minimal vertex cover of $H$, and $U$ does not involve the variables in $V\setminus V'.$
\end{proof}

\begin{remark}\label{r.J(H') subset J(H)}
	From Lemma \ref{l.gens Shad are gens Hyper}, we have  $J(H')=K[V']\cap J(H)$. Thus, each element $m$ in $J(H')$ also belongs to $J(H)$.
\end{remark}
As a consequence of Lemma \ref{l.gens Shad are gens Hyper}, we get the following result.
\begin{lemma}\label{l.m not in J(H)^s}
	If $(J(H')^s:m)=\mathfrak{p}\neq  (1)$ for some prime ideal $\mathfrak{p}$, then $m \notin J(H)^s$.
\end{lemma}
\begin{proof}
	Suppose $m \in J(H)^s$, then $m=m_1\cdots m_sM$ where the $m_i$'s  are monomial minimal generators of $J(H)$. Since $m$ only contains the variables in $V'$, each $m_i$ will also have this property. That means, $m_i \in J(H')$ for all $i \in \{1,2,\ldots,s\}$. Hence $m \in J(H')^s$, which contradicts $(J(H')^s:m)\neq  (1)$.
\end{proof}

	The next results show the first evidences that our construction really serves our purpose. We strongly use the classification in Proposition \ref{Corollary 3.5 cite{FHVT2}} and assume the existence of a graph $G\in \mathcal{S}(H)$.
	Then, we show that $J(H)^2$  only has associated primes inherited from $J(G)^2$. 
	The following lemma can be deduced from Corollary 3.4 in \cite{FHVT2}. We also include a proof for the convenience of the reader.
 
\begin{lemma}\label{l.oddcycle}Let  $C_{2n+1}=(V,E)$ be an $(2n+1)$-cycle. Then $(J(C_{2n+1})^2:x_{V})=\mathfrak p_V$.
\end{lemma} 
\begin{proof} A minimal cover of $C_{2n+1}$ involves exactly $n+1$ vertices. Then $J(C_{2n+1})^2$ is generated in degree $2n+2$ and $x_V\notin J(C_{2n+1})^2.$ Moreover, $x_1x_V=x_{\{1,2,4,\ldots, 2n \}}\cdot x_{\{1,3,5,\ldots, 2n+1 \}}\in J(C_{2n+1})^2$. Analogously, we get  $x_ix_V\in J(C_{2n+1})^2 $ for each $x_i\in V.$
\end{proof}
 
\begin{theorem}\label{t.oddcycle}
	Let $H=(V,E)$ be a hypergraph. If $G\in \mathcal{S}(H)$ is an odd cycle (i.e., $G=C_{2n+1}$ for some positive integer $n$), then $\mathfrak p_V\in \ass J(H)^2$. 
\end{theorem}

\begin{proof}
	Let $E = (e_1, \ldots, e_k)$. Since $G=(V',E')\in\mathcal S(H)$, the edges of $G$ are given by $\lbrace{e_1', \ldots, e_{k}'}\rbrace$ where $e_i' = e_i\cap V'$. By hypothesis, $G$ is an odd cycle, so $k=2n+1$ for some positive integer $n$.
	Without loss of generality, we relabel the vertices of $G$ so that 
\[ e_i'= \left\{ \begin{array}{ll}
     \lbrace{x_i, x_{i+1}}\rbrace, & \mbox{if $1\leq i \leq 2n$},\\
        \lbrace x_{2n+1},x_1\rbrace, & \mbox{if $i=2n+1$}.
    \end{array} \right. \] 
    
    %Now, let $Y:=V\setminus V'= \lbrace{y_1, \ldots, y_m}\rbrace$. We claim that $P_V=(x_1,\ldots, x_{2n+1}, y_1, \ldots, y_m)\in \ass J(H)^2$. 
    From Proposition \ref{Corollary 3.5 cite{FHVT2}}, we know that $\mathfrak p_{V'}\in \ass J(G)^2$, and Lemma \ref{l.oddcycle} we have $(J(G)^2:x_{V'})=\mathfrak p_V'$, where $x_{V'}=\prod_{i=1}^{2n+1}x_i$. 
    Then, we claim that $(J(H)^2:x_{V'})=\mathfrak p_V$. If $x_j\in V'$, $x_j x_{V'} \in J(G)^2\subseteq J(H)^2$. So $x_j\in (J(H)^2:x_{V'})$. Moreover, if $y_j\in V\setminus V'$, then there exists an edge $e_i\in E$ such that $y_j \in e_i$.
	Without loss of generality, one can assume that $i=1$.
	Thus we have that
\begin{align*}
	y_j x_{V'} = y_j x_1x_2\cdots x_{2n+1} = (y_jx_3x_5\cdots x_{2n+1})(x_1x_2x_4\cdots x_{2n}).
\end{align*}
	The right hand side of the above equality is in $J(H)^2$ since it is the product of two vertex covers of $H$. Thus, $y_j\in (J(H)^2: x_{V'})$. Finally, $x_{V'}\notin J(H)^2$ since $x_{V'}\notin J(H')^2.$  
\end{proof}

\begin{example} 
	Let $H$ be the hypergraph in Example \ref{e.triangles}. Since, for instance, the shadow of $H$ on $\{x_1,x_2,x_4\}$ is an odd cycle, we can state that
	$$\mathfrak p_V=(x_1,x_2,x_3,x_4,x_5)\in \ass(J(H)^2).$$
\end{example}

	Now we show that Theorem \ref{t.oddcycle} works in a more general setting. We need some further notation.
	Let $H=(V,E)$ be a hypergraph and let $G=(V',E')\in \mathcal S(H)$ be a graph. For a subset $U\subset V'$, we denote by 
	$$\widehat U:= \bigcup_{e\in E, e'\subseteq U} e  \subseteq V.$$
	%In particular, the corresponding set for a cycle $C_t=\{x_1, x_2,\ldots, x_t,x_{t+1}=x_1\}$ is $\hat C_t$ given by the union on the edges $e\in E$ containing  $\{x_{j}, x_{j+1}\}$ for some $j$. By the definition of a shadow, each such $e$ is unique.  
%	The following corollaries are consequences of Lemma \ref{l.localization}, Lemma \ref{l.edges} and  .

\begin{corollary}\label{c.oddcyle1}
	Let $H$ be a hypergraph and $H'$ a shadow of $H$. If $C_{2n+1}$ is an odd cycle that is a subhypergraph of  $H'$,
	then $\mathfrak p_{\widehat C_{2n+1}}\in \ass (J(H)^2)$.
\end{corollary}
\begin{proof}
	Say $H'=(V',E')$. We take the subhypergraph $\tilde H:= H_{\widehat C_{2n+1}}$ of $H$ on the vertex set $\widehat C_{2n+1}.$ 
%	We have the following facts
%\begin{itemize}
%	\item $\tilde H$ only contains the edges of $H$;
%	\item $\tilde H$ has a shadow on $V'$.
%\end{itemize}
	Notice that $\tilde H$ has a shadow on $C_{2n+1}$. That is the odd cycle $C_{2n+1}$.
	Thus, from Proposition \ref{Corollary 3.5 cite{FHVT2}} and Theorem \ref{t.oddcycle}, $\mathfrak p_{\widehat C_{2n+1}}\in \ass (J(\tilde H)^2)$. % and let $m\in K[V'\cap C_{2n+1}]$ be a monomial such that $(J(\tilde H')^2: m) = \mathfrak p_{C_{2n+1}}.$ Following the last part of the proof of Theorem \ref{t.oddcycle} we have $m=x_1\cdots x_{2n+1}$
	%and    $(J(\tilde H)^2: m) \supseteq \mathfrak p_{\hat C_{2n+1}}$, therefore $\mathfrak p_{\hat C_{2n+1}}\in \ass(J(\tilde H)^2)$.
	Moreover, from Lemma \ref{l.localization}, we have $\mathfrak p_{\widehat C_{2n+1}}\in \ass (J(H)^2)$.
	%Thus $\mathfrak p_{\hat C_{2n+1}}$ is an associated prime of the square of the cover ideal of a subhypergraph of $H$, and then $\mathfrak p_{\hat C_{2n+1}}\in \ass J(H)^2$.
\end{proof}	

\begin{corollary}\label{c.oddcyle2}
	Let $H$ be a hypergraph and $\tilde H$  a subhypergraph of $H$. If an odd cycle $C_{2n+1}\in \mathcal S(\tilde H)$, then $\mathfrak p_{\widehat C_{2n+1}}\in \ass ((J(H)^2)$.
\end{corollary}

\begin{example} \label{e.shadow} Let $H=(V,E)$ be the hypergraph with the vertex set $$V=\{x_1,x_2,x_3,x_4,x_5,x_6,x_7,x_8\}$$ and the edge set 
	$$E=\{ \{x_1,x_2,x_6\}, \{x_2,x_3,x_6\},\{x_3,x_4,x_8\},\{x_4,x_5,x_6\},\{x_1,x_5,x_7\} \}.
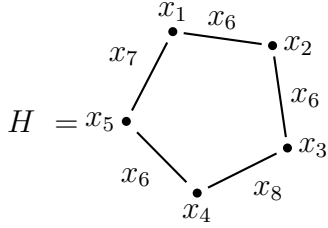
$$
\begin{center}
			\begin{tikzpicture}[scale=0.6]
		\node (f1) at (-3.3,1.59) {$H \ =\ $};
		%	\onslide<4>{	  \node (f1) at (-3.3,1.59) {$H' \ =\ $};}	
		\node (v4) at (0,0) {};
		\node (v3) at (2,1) {};
		\node (v2) at (1.67,3.27) {};
		\node (v1) at (-0.54,3.58) {};
		\node (v5) at (-1.57,1.59) {};
		\node (m1) at (0.56,3.39)[above]{$x_6$};
		\node (m2) at (1.83,2.11) [right]{$x_6$};
	    \node (m3) at (1,0.5) [below right]{$x_8$};
		\node (m4) at (-0.78,0.8) [below left]{$x_6$};
		\node (m5) at (-1.05,2.59) [above left]{$x_7$};
		\draw [thick] (v1)--(v2);
		\draw [thick] (v3)--(v2);
		\draw [thick] (v3)--(v4);
		\draw [thick] (v4)--(v5);
		\draw [thick] (v1)--(v5);
		
		\fill (v1) circle (0.1) node [above] {$x_1$};
		\fill (v2) circle (0.1) node [right] {$x_2$};
		\fill (v3) circle (0.1) node [right] {$x_3$};
		\fill (v4) circle (0.1) node [below] {$x_4$};
		\fill (v5) circle (0.1) node [left] {$x_5$};  
		\end{tikzpicture}
\end{center}

	The shadow of $H$ on the vertex set $V'=\{x_1,x_2,x_3,x_4,x_5\}$
	is
	$$H'=(V', \{\{x_1,x_2\}, \{x_2,x_3\},\{x_3,x_4\},\{x_4,x_5\},\{x_1,x_5\}\})\in \mathcal S(H).$$
	
	\begin{center}
		\begin{tikzpicture}[scale=0.6]
		\node (f1) at (-3.3,1.59) {$H' \ =\ $};
		%	\onslide<4>{	  \node (f1) at (-3.3,1.59) {$H' \ =\ $};}	
		\node (v4) at (0,0) {};
		\node (v3) at (2,1) {};
		\node (v2) at (1.67,3.27) {};
		\node (v1) at (-0.54,3.58) {};
		\node (v5) at (-1.57,1.59) {};
		\node (m1) at (0.56,3.39)[above]{ };
		\node (m2) at (1.83,2.11) [right]{ };
		\node (m3) at (1,0.5) [below right]{ };
		\node (m4) at (-0.78,0.8) [below left]{ };
		\node (m5) at (-1.05,2.59) [above left]{ };
		\draw [thick] (v1)--(v2);
		\draw [thick] (v3)--(v2);
		\draw [thick] (v3)--(v4);
		\draw [thick] (v4)--(v5);
		\draw [thick] (v1)--(v5);
		
		\fill (v1) circle (0.1) node [above] {$x_1$};
		\fill (v2) circle (0.1) node [right] {$x_2$};
		\fill (v3) circle (0.1) node [right] {$x_3$};
		\fill (v4) circle (0.1) node [below] {$x_4$};
		\fill (v5) circle (0.1) node [left] {$x_5$};  
		\end{tikzpicture}
	\end{center}
	
	We see that $H'$ is a graph, precisely it is an odd cycle of length 5.
	By Theorem \ref{t.oddcycle}, we have that
	$$\mathfrak p_V=(x_1,x_2,x_3,x_4,x_5,x_6,x_7,x_8)\in \ass(J(H)^2).$$
	
	Now, we take the shadow of $H$ on the vertex set $V''=\{x_1,x_3,x_5,x_6,x_8\}$. 
    The shadow of $H$ on $V''$ is
    $$H''=(V'', \left\{ \{x_1,x_6\},\{x_3,x_6\},\{x_3,x_8\},\{x_5,x_6\},\{x_1,x_5\} \right\})\in \mathcal S(H).$$
    
    \begin{center}
    	\begin{tikzpicture}[scale=0.7]
    	\node (f1) at (-3.3,1.59) {$H'' \ =\ $};
    	\node (v4) at (0,0) {};
    	\node (v3) at (2,1) {};
    	\node (v2) at (1.67,3.27) {};
    	\node (v1) at (-0.54,3.58) {};
    	\node (v5) at (-1.57,1.59) {};
    	\draw [thick] (v1)--(v2);
    	\draw [thick] (v3)--(v2);
    	\draw [thick] (v3)--(v4);
    	\draw [thick] (v2)--(v5);
    	\draw [thick] (v1)--(v5);
    	
    	\fill (v1) circle (0.1) node [above] {$x_1$};
    	\fill (v2) circle (0.1) node [right] {$x_6$};
    	\fill (v3) circle (0.1) node [right] {$x_3$};
    	\fill (v4) circle (0.1) node [below] {$x_8$};
    	\fill (v5) circle (0.1) node [left] {$x_5$};  
    	\end{tikzpicture}
    \end{center}
    
    Note that $H''$ has a subhypergraph that is a cycle of length 3, $C_3=\{\{x_1,x_6\},\{x_5,x_6\},\{x_1,x_5\}\}.$ By Corollary \ref{c.oddcyle1}, this cycle produces an element in $\ass(J(H)^2).$ So, we get
    $$\mathfrak p_{\widehat C_3}= (x_1,x_2,x_4,x_5,x_6,x_7)\in \ass (J(H)^2).$$
\end{example}

In the following example we show that condition $(b)$ in Definition \ref{d.shadow} is necessary for Theorem \ref{t.oddcycle}. 
\begin{example}
 Let $H=(V,E)$ be the hypergraph with the vertex set $$V=\{x_1,x_2,x_3,x_4,x_5,x_6,x_7,x_8,x_9\}$$ and the edge set 
 \begin{align*}
 E=\{&\{x_1, x_2,x_6, x_8\}, \{x_2,x_3,x_8,x_6\},\{x_3,x_4,x_7,x_9\},\\
 &\{x_4,x_5,x_6,x_8\},\{x_1,x_5,x_7,x_9\},\{x_8,x_9\}, \{x_6,x_7\} \}.
 \end{align*}
 A Macaulay2 computation \cite{MACAULAY2} shows that $$\ass (J(H)^2)= \{\mathfrak p_e \ |\ e\in E \}.$$ 
 
We claim that ignoring the rule $|E|=|E'|$ in the condition (b) of Definition \ref{d.shadow} Theorem \ref{t.oddcycle} does not hold. Indeed, we get on the vertex set $V'=\{x_1,x_2,x_3,x_4,x_5\}$ the hypergraph
 $$H'=(V', \{\{x_1,x_2\}, \{x_2,x_3\},\{x_3,x_4\},\{x_4,x_5\},\{x_1,x_5\}\})\in \mathcal S(H).$$
 That is an odd cycle and, see Lemma \ref{l.oddcycle}, we have
 $$\mathfrak p_{V'}=(x_1,x_2,x_3,x_4,x_5)\in \ass(J(H')^2).$$
\end{example}

In the last part of this section we prove that, under some suitable hypothesis,  all the associated primes of $J(H)^2$ come from some non-trivial shadow  (we will see in Proposition \ref{p.associated square}). We need an auxiliary lemma.

\begin{lemma}\label{l.for square}
	Let $H=(V,E)$ be a hypergraph, and suppose $(J(H)^{s}: m) = \mathfrak p_{V}$ for some monomial $m$.  Let $V'\subsetneq V$ be a proper subset such that $e_i\cap e_j\subseteq V'$ for each $e_i, e_j\in E$,  $i\neq j$.  Then $y^{s-1}$ does not divide $m$ for each $y\in V\setminus V'$. 
\end{lemma}
\begin{proof}
	Let $y$ be an element in $V\setminus V'$. We write $m=y^am'$, where, unless to rename, $y\in e_1$ and $y$ does not divide $m'.$
	If $a\ge s$ then, since $ym\in J(H)^s,$ we get $ym=m_1 \cdots m_sM$,  where $m_j$ corresponds to a minimal vertex cover of $H$ for $j \in \lbrace 1,\ldots, s\rbrace$. Thus 
	$y$ divides $M$ and $m=m_1 \cdots m_s(M/y).$ This contradicts $m\notin J(H)^s$. 
	Therefore, we can assume $a=s-1.$ 
	We work by induction on $r=|e_1\setminus V'|$. If $r=1$, i.e., $ e_1=(e_1\cap V')\cup \{y\}$, then from $ym\in J(H)^s,$ we get $ym=(ym_1) \cdots (ym_s)M$, where $ym_j$ are minimal vertex covers of $H$.
	Note that, for each $x_j\in e_1\cap V'$, we can see that $x_j$ does not divide $m_1,\ldots, m_s$ (these are minimal vertex covers) and $x_j$ does not divide $M$ (otherwise we can just delete $y$ and get $m\in J(H)^s$). This implies that $m\notin (\mathfrak p_{e_1})^s.$ To get a contradiction, we just take some $z\notin e_1$ and remember that by hypothesis $zm\in J(H)^s$ but $zm\notin (\mathfrak p_{e_1})^s$.    
	
	If $r>1$, i.e., $ e_1=(e_1\cap V')\cup \{y_1,\ldots, y_r\}$, then just note that $V''=V'\cup \{y_1,\ldots, \widehat{y_{i}},\ldots, y_{r}\}$ satisfies the hypothesis of the theorem and $ e_1=(e_1\cap V'')\cup \{y_i\}$. 
\end{proof}

\begin{proposition}\label{p.associated square} Let $H=(V,E)$ be a hypergraph and  $H'=(V',E')\in \mathcal S(H)$ a shadow of $H$. Assume that $e_i\cap e_j\subseteq V'$ for each $e_i, e_j\in E$, where $i\neq j.$  If $\mathfrak p_{V}\in \ass J(H)^{2}$,  then 	$\mathfrak p_{V'}\in \ass J(H')^2.$    
\end{proposition}
\begin{proof}
By the definition of associated primes, there exists a monomial $m\in K[V]$ such that $(J(H)^2:m)=\mathfrak p_V.$ Say $V'=\{x_1,\ldots,x_a \}$ and $V\setminus V'=\{y_1,\ldots,y_b \}$. 
By Lemma \ref{l.for square} $y_j$ does not divide $m$ for $j=1,\ldots, b.$ Then $m\in K[V']$ and therefore $(J(H')^2:m)=\mathfrak p_{V'}.$
\end{proof}

\section{A first case}\label{s.main case}

In this section we investigate the relations between a hypergraph and its shadows in a particular case of study. Precisely, we consider shadows that only differ from the starting hypergraph by one edge and one vertex. 

In this section, we shall use the following notation.
\begin{notation}\label{n.notation}
	Let $H=(V,E)$ be a hypergraph and $H'=(X,E')$ a shadow of $H$ such that
	\begin{itemize}
		\item[(a)] $X=\{x_1,\ldots,x_n\}$ and $V=X\cup\{y\}$; and
		\item[(b)] $y$ only belongs to one edge, say  $e_y\in E$. 
	\end{itemize}
	After renaming, say $e_y=\{x_1,\ldots,x_t,y\}$. We set $e:=e'_y= \{x_1,\ldots,x_t\}$, then we have $H'=\left\{X,(E\setminus \{e_y\})\cup \{e\} \right\}$.
	Moreover, to shorten the notation, $\tilde H$ will denote the subhypergraph of $H$ on $X$.
	%(i.e. $H_X$. : we don't need this? )
	We denote by $\mathfrak p_e$ and $\mathfrak p_{e_y}$ the prime ideals generated by the variables in $e$ and $e_y$ respectively. 
\end{notation}
We remark that, in this setting, the hypergraphs $\tilde H$ and $H'$ share the vertex set $X$. Moreover, they share the same edges except for $e$. 
%Moreover, each element in $E\setminus \{e_y\}$ is an edge for both $\tilde H$ and $H'$.

We will abuse notation: given a subset $F\subseteq X\subseteq V$, we will write $\mathfrak{p}_F$ to denote both the ideals in $K[X]$ and in $K[V].$ 

Here, we anticipate the results of this section.  
In the first part of the section, we investigate the relation linking associated primes of $J(\tilde H)^s$ and $J(H')^s$ with the elements in $\ass (J(H)^s)$. We have seen in Corollary \ref{c.prime of subhypergraph} that  if $\mathfrak p\in \ass (J(\tilde H)^s)$ then $\mathfrak p\in \ass (J(H)^s).$ What about the associated prime of $J( H')^s$?  We will show that if $\mathfrak p\in \ass (J( H')^s)$, then either $\mathfrak p+(y)\in \ass (J(H)^s)$ or  $\mathfrak p\in \ass (J(H)^s)$. This depends on a further condition of a monomial $m$ such that $(J(H')^s: m)=\mathfrak p.$ The following diagram summarizes these results.
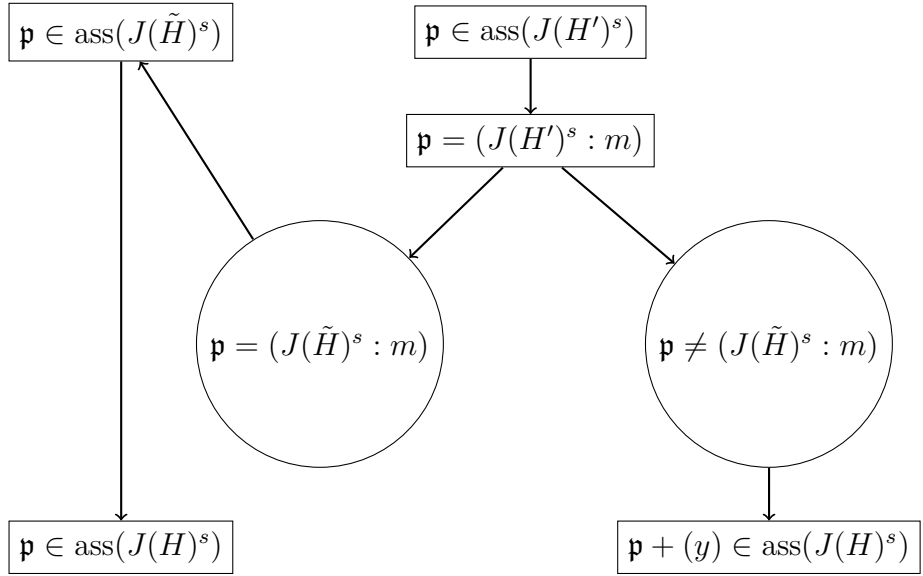
\begin{center}
 \vspace*{-5pt}
		\begin{tikzpicture}[scale=0.6]
	\matrix [column sep=-5mm, row sep=7mm] {
  	\node (pHt) [draw, shape=rectangle] {$\mathfrak p\in \ass (J(\tilde H)^s) $}; & & \node (pHs) [draw, shape=rectangle] {$\mathfrak p\in \ass (J(H')^s )$};& \\
		& & \node (m) [draw, shape=rectangle] {$\mathfrak p=(J(H')^s: m) $};& \\
		& 	\node (mYes) [draw, shape=circle] {$\mathfrak p= (J(\tilde H)^s: m )$}; & & 	\node (mNo) [draw, shape=circle] {$\mathfrak p\neq (J(\tilde H)^s: m )$};  \\
		\node (pH) [draw, shape=rectangle] {$\mathfrak p\in \ass (J(H)^s) $};	& 	& &  \node (pyH) [draw, shape=rectangle] {$\mathfrak p+(y)\in \ass (J(H)^s )$}; \\
	};
	\draw[->, thick] (pHt) -- (pH);
	\draw[->, thick] (pHs) -- (m);
	\draw[->, thick] (m) -- (mYes);
	\draw[->, thick] (m) -- (mNo);
	\draw[->, thick] (mYes) -- (pHt);
	\draw[->, thick] (mNo) -- (pyH);
	\end{tikzpicture}
	\vspace*{-5pt}
\end{center}
In the second part of the section, we will reverse the investigation. Starting from a prime associated to  $J(H)^s$, we will look for which conditions allow us to find a relation with an element in $J(\tilde H)^s$ or $J(H')^s$. Precisely, if $\mathfrak p \in \ass (J(H)^s)$ and $y\notin \mathfrak p$ then $\mathfrak p\in \ass (J(\tilde H)^s)$. Moreover, if $\mathfrak p= (y)+\mathfrak p'$, it seems natural to ask if $\mathfrak p'\in \ass (J(H')^s),$ which we positively answer under an extra (restrictive) condition.  In the next section, see \ref{e.persistence does not fail}, we will show that not all the primes  $(y)+\mathfrak p'$ associated to $J(H)^s$ come from a prime $\mathfrak p'$ in the shadow.

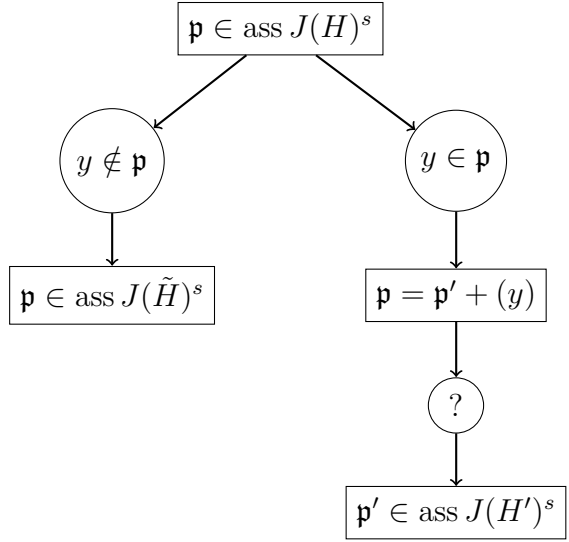
\begin{center}
	\begin{tikzpicture}
	\matrix [column sep=-5mm, row sep=7mm] {
		& \node (pH) [draw, shape=rectangle] {$\mathfrak p\in \ass J(H)^s $}; &  \\
		\node (NOy) [draw, shape=circle] {$y\notin \mathfrak p$}; & & \node (YESy) [draw, shape=circle] {$y\in \mathfrak p$};	 \\
		\node (pHt) [draw, shape=rectangle] {$\mathfrak p\in \ass J(\tilde H)^s $}; & &
		\node (pTOq) [draw, shape=rectangle] {$\mathfrak p= \mathfrak p' +(y)$ }; \\
		& & \node (qd) [draw, shape=circle] {?}; \\
		& & \node (pHs) [draw, shape=rectangle] {$\mathfrak p'\in \ass J(H')^s $}; \\
	};
	\draw[->, thick] (pH) -- (NOy);
	\draw[->, thick] (NOy) -- (pHt);
	\draw[->, thick]  (YESy)-- (pTOq);
	\draw[->, thick] (pH) -- (YESy);
	\draw[->, thick] (pTOq) --  (qd);
	\draw[->, thick] (qd) --  (pHs);
	\end{tikzpicture}
\end{center}

We start with an auxiliary result.
\begin{lemma}
Let $m \in \mathcal G(J(H))$ be a monomial minimal generator of $J(H)$. If $y|m$, then $x_i \not| m$ for all  $x_i\in e$.
\end{lemma}
\begin{proof} In our setting, $y$ only belongs to the edge $e_y=\{x_1,\ldots,x_t,y\}$. 
	Since $m$ is a minimal vertex cover of $H$, if $x_i \in e= \{x_1,\ldots,x_t\}$  divides $m$, then $\dfrac{m}{y}$ is also a vertex cover. This contradicts the minimality of $m$.
\end{proof}

In order to relate the associated primes of $J(H')^s$ to the associated primes of $J(H)^s,$ the following proposition will be crucial.

\begin{proposition}\label{p.on associated to H'}
	Let $(J(H')^s:m)=\mathfrak{p}_F$ be a prime ideal, for some $F\subseteq X.$ Then,
	$$(J(H)^s:m)= \mathfrak p_F + \mathfrak q$$ where $\mathfrak q \subseteq (y).$ In other words, no monomial only involving the variables in $X\setminus F$  belongs to $(J(H)^s:m)$.
\end{proposition}
\begin{proof} 
	Say $F:=\{ x_{i_1}, \ldots, x_{i_k}\}$ and $\{x_{\ell_1},\ldots, x_{\ell_r}\}=X\setminus F.$ Recall that $e=\{x_1,\ldots, x_t\}.$
	First we show that $\mathfrak{p}_F\subseteq(J(H)^s:m)\subsetneq (1).$
	From Lemma \ref{l.m not in J(H)^s} we have $m \notin J(H)^s$ and then $(J(H)^s:m) \neq (1)$. By hypothesis, for each $x_j \in F$ we have $x_jm\in J(H')^s$ i.e. $m=m_1\cdots m_s M$ for some monomials $m_i\in J(H')\subseteq K[X]$. But these monomials, see Remark \ref{r.J(H') subset J(H)} also belongs to $J(H)$. Hence, $x_jm \in J(H)^s$ and $\mathfrak{p}_F \subseteq (J(H)^s:m)\subseteq K[V]$.
	
	In order to conclude the proof, take any monomial $x^{a_1}_{\ell_1} \cdots x^{a_t}_{\ell_t}$ in variables in $X\setminus F$. Suppose that $x^{a_1}_{\ell_1} \cdots x^{a_t}_{\ell_t}m=m_1\cdots m_sM \in J(H)^s$, where the $m_j$'s are minimal generators of $J(H)$ in the variables in $X$. The monomials $m_j \in J(H')$ and then $x^{a_1}_{\ell_1} \cdots x^{a_t}_{\ell_t} \in (J(H')^s:m)=\mathfrak{p}_F$, which is a contradiction.
\end{proof}

\begin{lemma} \label{p.tildeH}
	Let $(J(H)^s:m)= \mathfrak p$ and $y\notin \mathfrak p$. Then $(J(\tilde H)^s:m)= \mathfrak p$. 
\end{lemma}
\begin{proof} Say $\mathfrak p=\mathfrak p_F$ for some $F\subseteq X.$ First note that $m\notin J(\tilde H)^s.$ Indeed, if  $m=m_1\cdots m_s \cdot M\in J(\tilde H)^s$ with $m_1,\ldots, m_s$ minimal vertex covers of $\tilde H$, then $y^sm\in J(H)^s$. This contradicts $(J(H)^s:m)= \mathfrak p$. We claim that $(J(\tilde H)^s:m)\supseteq \mathfrak p.$ Indeed, if $x_j\in F$, then $x_jm\in J(H)^s\subseteq J(\tilde H)^s.$
	In order to obtain the assertion, we take a monomial $T\notin \mathfrak p_F$ and assume that $Tm\in J(\tilde H)^s$. Again from $Tm=m_1\cdots m_s \cdot M\in J(\tilde H)^s$ with $m_1,\ldots, m_s$ minimal vertex covers of $\tilde H$, we get $Ty^s\in (J(H)^s:m)$ which contradicts the hypothesis.
\end{proof}

\begin{theorem}\label{t.tildeH} Let $(J(H')^s:m)= \mathfrak p$. Then, we have 
\begin{itemize}
	\item[(a)] $(J(H)^s:m) = \mathfrak p$ if and only if $(J(\tilde H)^s:m)= \mathfrak p$;
	\item[(b)] 	$(J(H)^s:m\cdot m_0) = \mathfrak p+(y)$, for some monomial $m_0\notin \mathfrak p,$ if and only if $(J(\tilde H)^s:m)\neq \mathfrak p$.%\mathfrak p+\langle y \rangle\in \ass (J(H)^s)
	
\end{itemize}
\end{theorem}
\begin{proof}
	Note that $y\notin \mathfrak p$, so one implication in (a) follows from Lemma \ref{p.tildeH}. %In order to prove the converse, 
	Set $\mathfrak p_F:=\mathfrak p=(J(\tilde H)^s:m)$ and say $X\setminus F=\{x_{\ell_1},\ldots, x_{\ell_r}\}.$ 
	
	By Proposition \ref{p.on associated to H'}, we have $(J(H)^s:m)=\mathfrak p+ \mathfrak q$ where either $\mathfrak q=(0)$  or  $\mathfrak q$ is minimally generated by monomials $y^a\cdot x_{\ell_1}^{a_1}x_{\ell_2}^{a_2}\cdots x_{\ell_r}^{a_r}$ for some $a>0$ and $a_1,\ldots, a_r\ge 0$.
	We claim that $\mathfrak q=(0)$. Indeed, if $T:=y^a\cdot x_{\ell_1}^{a_1}x_{\ell_2}^{a_2}\cdots x_{\ell_r}^{a_r}\in \mathfrak q$, we get $\dfrac{T}{y^a}\in (J(\tilde H)^s:m)=\mathfrak p_F$ which contradicts the hypothesis. 
		
	Now we prove item (b). With the notation as above, we have $(J(H)^s:m)=\mathfrak p+ \mathfrak q$. 
	First assume  $(J(\tilde H)^s:m)\neq \mathfrak p$. Then $\mathfrak q$ is not the zero ideal. Consider the non-empty set $$\{b\in \mathbb{N} \mid y^b\ \text{ divides }\ M\ \text{ for some }\ M\in \mathfrak q\},$$
	and let $a$ be its minimum element.
	Let $T:=y^a\cdot x_{\ell_1}^{a_1}x_{\ell_2}^{a_2}\cdots x_{\ell_r}^{a_r}\in \mathfrak q$ be a monomial minimal generator in $\mathfrak q$. 
	We collect some relevant facts:
\begin{itemize}
	\item $a>0,$ by Proposition \ref{p.on associated to H'};
	\item $m\dfrac{T}{y}\notin J(H)^s$, by the minimality of $T;$
	\item $x_{\ell_1}^{a_1}x_{\ell_2}^{a_2}\cdots x_{\ell_r}^{a_r}\cdot m\dfrac{T}{y}\notin J(H)^s$, by the minimality of $a$;
	\item $y \cdot m\dfrac{T}{y}=mT \in J(H)^s$.
\end{itemize} 
Then, we get 
$\left(J(H)^s: m\dfrac{T}{y}\right)=\mathfrak p+ (y),$ and $\mathfrak p+(y)\in \ass (J(H)^s).$ 

Vice versa, assume $(J(H)^s:m\cdot m_0) = \mathfrak p+(y)$, for some monomial $m_0\notin \mathfrak p.$ So, we have $ymm_0\in J(H)^s$ and say $ymm_0=ym_1\cdot m_2\cdots m_s\cdot M\in J(H)^s$ with $ym_1,\ldots, m_s$ corresponding to minimal vertex covers of $H$. 
Then, we get $mm_0=m_1\cdot m_2\cdots m_s\cdot M\in J(\tilde H)^s$, i.e., $m_0 \in (J(\tilde H)^s:m)$. Since $m_0$ does not involve the variables in $\mathfrak p$, we get a contradiction. 
\end{proof}

In particular, the next result shows that item (a) in Theorem \ref{t.tildeH} is always satisfied if  $\mathfrak p_e\not\subseteq \mathfrak p$.
%\begin{proposition}\label{p.from H' to tilde H}
%	Let $\mathfrak p\in \ass (J(H')^s)$ and $\mathfrak p_e\not\subseteq \mathfrak p$. Then $\mathfrak p\in \ass (J(\tilde H)^s)$.   
%\end{proposition}
%\begin{proof} Let $m\in K[X]$ be a monomial such that $(J(H')^s:m) = \mathfrak p$.  Since $J(H')\subseteq J(\tilde H)$ we have $(J(\tilde H)^s:m) \supseteq \mathfrak p.$ Say $(J(\tilde H)^s:m) = \mathfrak p+ \mathfrak q$ and assume $\mathfrak q\neq (0)$.  For a monomial $m'\in \mathfrak q\setminus \mathfrak p$ we have $mm'\in J(\tilde H)^s$, i.e., $mm'= m_1\cdots m_s$ is a product of $s$ covers of $\tilde H$. 
%Thus, for a variable $x_j\in \mathfrak p_e\setminus \mathfrak p$, note that $m'x_j^s\notin \mathfrak p$ and $mm'x_j^s= (x_jm_1)\cdots (x_jm_s) \in J(H')^s$.  Since $x_j^sm'\in (J(H')^s:m)= \mathfrak p$ we get a contradiction. 
%\end{proof}

\begin{proposition}\label{p.shadow1} Let $(J(H')^s:m)= \mathfrak p$. If $\mathfrak p_e \not\subseteq \mathfrak p$, then $(J(H)^s:m)= \mathfrak p$.
\end{proposition}
\begin{proof}
	Say $\mathfrak p= \mathfrak p_F$ with $F:=\{ x_{i_1}, \ldots, x_{i_k}\}$ and $\{x_{\ell_1},\ldots, x_{\ell_r}\}=X\setminus F.$
	%First assume $\mathfrak p_e \not\subseteq \mathfrak p=(x_{i_1},\ldots,x_{i_k})$. 
	By Proposition \ref{p.on associated to H'} we have $(J(H):m)=\mathfrak p+ \mathfrak q$ where $\mathfrak q$ is an ideal minimally generated by monomials which are not only in variables  $\{x_{\ell_1},\ldots, x_{\ell_r}\}= X\setminus F$; i.e., a minimal generator of $\mathfrak q$ is a monomial $y^bx_{\ell_1}^{a_1}\cdots x_{\ell_r}^{a_r}$ for some $a_1, \ldots, a_r\ge 0$ and $b>0$. 
	Assume on the contrary that  $\mathfrak q\neq 0$. 
	Take any  minimal generator in $\mathfrak q$, say $T:=y^bx_{\ell_1}^{a_1}\cdots x_{\ell_r}^{a_r}$. %Note $\{x_{\ell_1},\ldots, x_{\ell_r}\}\subseteq \{x_1,\ldots x_n\}\setminus \{x_{i_1},\ldots,x_{i_k} \}.$
	Then $m\cdot T= m_1\cdots m_s\cdot M\in J(H)^s $ where the $m_i$'s are minimal vertex covers of $H.$ Note that $y$ does not divide $M$. Otherwise,  we get  $x_{\ell_1}^{a_1}x_{\ell_2}^{a_2}\cdots x_{\ell_r}^{a_r}y^{b-1}\in (J(H)^s:m)$, contradicting the minimality of $T.$
	Then we can write, after relabeling, $m_{i}=ym_{i}'$ for $i=1, \ldots b,$ $m\cdot T= (ym_1')\cdots(ym_b')\cdot m_{b+1} \cdots  m_s\cdot M\in J(H)^s.$ %The minimality of $m_i$ implies $x_j$ does not divide $m_i'$ for $i=1, \ldots b$ and $x_j\in e.$  
	Say $x_1\in \mathfrak p_e$ and $x_1\notin \mathfrak p$, then we get  $$m\cdot T\dfrac{x_1^b}{y^b}= (x_1m_1')\cdots(x_1m_b')\cdot m_{b+1} \cdots  m_s\cdot M\in J(H)^s.$$	
	But $m\cdot T\dfrac{x_1^b}{y^b}$ only contains variables of $X$. Then $T\dfrac{x_1^b}{y^b}\in (J(H'):m)= \mathfrak p.$ By Proposition \ref{p.on associated to H'}, this is a contradiction since $T\dfrac{x_1^b}{y^b}$ only contains variables not in $\mathfrak p.$ 
\end{proof}

Recall that by Corollary \ref{c.prime of subhypergraph}, a prime associated to  $J(H)^s$ either belongs to $\ass (J(\tilde H)^s)$ or it contains the variable $y$. This is summarized in the following statement.
\begin{corollary}\label{c.dec ass J(H)^s} We have
	$$\ass (J(H)^s) = \ass (J(\tilde H)^s) \cup \mathcal A,$$ where if $\mathfrak p\in \mathcal A$, then $y\in \mathfrak p.$
\end{corollary}

\begin{question}\label{q.shadow}
	Do the elements in $\mathcal A$, mentioned in Corollary \ref{c.dec ass J(H)^s}, all come from the shadow?  More precisely, 	if $\mathfrak p=\mathfrak p'+(y)\in \ass J(H)^s$, then is $\mathfrak p'\in \ass J(H')^s$?  
\end{question}

We will show in the next section, see Example \ref{e.persistence does not fail}, that such question has in general a negative answer.
By the way, in the next theorem, we positively answer this question under a suitable condition.

\begin{theorem}\label{t.from H to H'}
	Let $\mathfrak p=\mathfrak p'+(y)\in \ass (J(H)^s)$. If $\mathfrak p\notin \ass (J(H)^s:y)$, then $\mathfrak p'\in \ass (J(H')^s)$.
\end{theorem}
\begin{proof}
Take the short exact sequence
$$ 0\to \dfrac{K[V]}{J(H)^s:y} \to \dfrac{K[V]}{J(H)^s} \to \dfrac{K[V]}{J(H)^s+(y)}\to 0.$$
From theorem 6.3 in \cite{Matsumura}  we have that
$$\Ass (K[V]/J(H)^s)\subseteq \Ass (K[V]/J(H)^s:y) \bigcup \Ass (K[V]/J(H)^s+(y)).$$
Denoted by $J'$ the cone ideal of  $J(H')^s$ in the ring $K[V]$, note that  $K[V]/J(H)^s+(y)=K[V]/J'+(y).$ 
Since, by hypothesis $\mathfrak p\in \Ass (K[V]/J(H)^s+(y))$, then $\mathfrak p \in \ass (J' +(y))$, i.e. $\mathfrak p' \in \ass (J(H')^s)$.
\end{proof}

\begin{remark}
	Question \ref{q.shadow} has a positive answer if $(J(H)^s:m)=\mathfrak p+(y)$ for some $m\in K[X].$ 
\end{remark}

In the next examples we show how to describe all the associated prime ideals of $J(H)^s$ from $\ass (J(\tilde H)^s)$ and $\ass (J(H')^s).$ 
\begin{example}
	Let $H$ be the hypergraph on the vertex set $V=\{x_1,x_2,x_3,x_4,x_5, y\}$ and the edge set 
	$$E=\{\{x_1,x_2\}, \{x_2,x_3\}, \{x_3,x_4\}, \{x_4,x_5\}, \{x_1,x_5\}, \{x_1,x_3,y\} \}.$$ 
	Set $X:=\{x_1,x_2,x_3,x_4,x_5\}.$ 
	Then the shadow of $H$ on $X:=\{x_1,x_2,x_3,x_4,x_5\}$ is
	$$H' =( X, \{\{x_1,x_2\}, \{x_2,x_3\}, \{x_3,x_4\}, \{x_4,x_5\}, \{x_1,x_5\}, \{x_1,x_3\}  \} ).$$ 
	Moreover, the subhypergraph of $H$ on $X$ is
	$$\tilde H =(X, \{\{x_1,x_2\}, \{x_2,x_3\}, \{x_3,x_4\}, \{x_4,x_5\}, \{x_1,x_5\}  \}).$$
	A Macaulay2 computation shows that	
	$$\ass J(H')^3=\{(x_1,x_2), (x_2,x_3), (x_3,x_4), (x_4,x_5), (x_1,x_5), (x_1,x_3)  \}\cup \{(x_1,x_2,x_3)\}$$
	and
	$$\ass J(\tilde H)^3=\{(x_1,x_2), (x_2,x_3), (x_3,x_4), (x_4,x_5), (x_1,x_5)\}\cup \{(x_1,x_2,x_3,x_4,x_5)\}. $$	 
	
	From Theorem \ref{t.tildeH}, we know that $(x_1,x_2,x_3,y)\in \ass (J(H)^3)$. 
	Moreover, one can check that
	
	$$\ass (J(H)^3)= \ass (J(\tilde H)^3)\cup \{(x_1,x_3,y),(x_1,x_2,x_3,y)\}.$$
	\end{example}

\begin{example}
	Let $H$ be the hypergraph on the vertex set $V=\{x_1,x_2,x_3,x_4,x_5,y\}$ given by
	$$H=(V,\{\{x_1,x_2,x_3\}, \{x_1,x_4\}, \{x_2,x_4\}, \{x_2,x_5\}, \{x_3,x_5\}, \{x_4,x_5,y\} \}).$$
	Set $X:=\{x_1,x_2,x_3,x_4,x_5\}$; then the shadow of $H$ on $X$ is 
	$$H' =\big( X, \{\{x_1,x_2,x_3\}, \{x_1,x_4\}, \{x_2,x_4\}, \{x_2,x_5\}, \{x_3,x_5\}, \{x_4,x_5\} \} \big).$$ 
	Moreover, the subhypergraph of $H$ on $X$ is
	$$\tilde H =\big(X, \{\{x_1,x_2,x_3\}, \{x_1,x_4\}, \{x_2,x_4\}, \{x_2,x_5\}, \{x_3,x_5\} \}\big ).$$
	
	Using Macaulay2, we compute that
	$$\ass J(H')^2=\{(x_1,x_2,x_3), (x_1,x_4), (x_2,x_4), (x_2,x_5), (x_3,x_5), (x_4,x_5) \}\cup$$
	$$\cup \{(x_1,x_2,x_3,x_4),(x_1,x_2,x_3,x_5),(x_2,x_4,x_5)\}$$
	and  $$\ass J(H')^3= \{(x_1,x_2,x_3), (x_1,x_4), (x_2,x_4), (x_2,x_5), (x_3,x_5), (x_4,x_5) \}\cup$$
	$$\cup \{(x_1,x_2,x_3,x_4),(x_1,x_2,x_3,x_5),(x_2,x_4,x_5)\}\cup$$
	$$\cup \{(x_1,x_2,x_3,x_4,x_5)\}.$$
	
	We also know that  $\ass (J(\tilde H)^3)$ and $\ass (J(\tilde H)^2)$ share the same elements, precisely
	$$\{(x_1,x_2,x_3), (x_1,x_4), (x_2,x_4), (x_2,x_5), (x_3,x_5) \}$$
	$$\cup \{(x_1,x_2,x_3,x_4),(x_1,x_2,x_3,x_5),(x_1,x_2,x_3,x_4,x_5)\}.$$	 
	
	Then, from Proposition \ref{p.associated square} and Theorem \ref{t.tildeH}, we have
	$$\ass (J(H)^2)=\ass (J(\tilde H)^2) \cup \{(x_4,x_5,y), (x_2,x_4,x_5,y)\}.$$

	What about $\ass (J(H)^3)$?
	The element $(x_1,x_2,x_3,x_4,x_5)$ appears both in $\ass (J(\tilde H)^3)$ and $\ass (J(H')^3)$ and it contains $(x_4,x_5)$. One can check that
	$$\ass (J(H')^3:m) =(x_1,x_2,x_3,x_4,x_5) $$
	and 
	$$\ass (J(\tilde H)^3:m) =(x_1,x_2,x_3,x_4,x_5)$$
	where $m:=x_1x_2^2x_3x_4^2x_5^2.$
	
	Thus, by Theorem \ref{t.tildeH}, $(x_1,x_2,x_3,x_4,x_5,y)\notin \ass (J(H)^3)$ and one can check that
	$$\ass (J(H)^3)= \ass (J(\tilde H)^3)\cup \{(x_4,x_5,y), (x_2,x_4,x_5,y)\}.$$
\end{example}

\begin{example}\label{e.p and p+y}
	Let $H$ be the hypergraph on the vertex set $V:=\{x_1,x_2,x_3,x_4,x_5,y\},$ given by
	$$H=\left(V, \{\{x_1,x_2\}, \{x_1,x_3\}, \{x_1,x_4\}, \{x_1,x_5,y\}, \{x_2,x_3,x_4,x_5\}\}\right).$$
	Let $H'$ be the shadow of $H$ on the vertex set $X:=\{x_1,x_2,x_3,x_4,x_5\}$ and $\tilde H$ the subhypergraph of $H$ on $X.$ Then 
	$$H'=\left(X, \{\{x_1,x_2\}, \{x_1,x_3\}, \{x_1,x_4\}, \{x_1,x_5\}, \{x_2,x_3,x_4,x_5\}\}\right)$$
	and 
	$$\tilde H=\left(X, \{\{x_1,x_2\}, \{x_1,x_3\}, \{x_1,x_4\},  \{x_2,x_3,x_4,x_5\}\}\right).$$
	
	A Macaulay2 computation shows that
	$$\ass J(H')^2=\{(x_1,x_2), (x_1,x_3), (x_1,x_4), (x_1,x_5), (x_2,x_3,x_4,x_5)\}\cup \{(x_1,x_2,x_3,x_4,x_5)\} $$
	and 
	$$\ass J(\tilde H)^2= \{(x_1,x_2), (x_1,x_3), (x_1,x_4),  (x_2,x_3,x_4,x_5)\}\cup \{(x_1,x_2,x_3,x_4,x_5)\}.$$
	The element $(x_1,x_2,x_3,x_4,x_5)$ appears both in $\ass (J(\tilde H)^3)$ and $\ass (J(H')^3)$ and it contains $(x_1,x_5)$. One can check that
	$$(J(H')^2:x_1x_2x_3x_4x_5) = (x_1,x_2,x_3,x_4,x_5)$$
		but $x_1x_2x_3x_4x_5\in J(\tilde H)^2.$ Thus
	
	$$\ass (J(H)^2)= \ass (J(\tilde H)^2)\cup \{(x_1,x_5,y), (x_1,x_2,x_3,x_4,x_5, y)\}.$$
	% $\{1,2,3,4,5\},\{1,2,3,4,5,6\}\subseteq \ass J(H)^2.$ 
%	Indeed

\end{example}

%\textbf{(G: what about add some picture?)}
\section{An application to the persistence property}\label{s.application}
In this section, we apply the results of Section \ref{s.main case} to the persistence problem.
A squarefree monomial ideal $I$ is said to have the persistence property if $\ass(I^s)\subseteq\ass(I^{s+1})$ for any integer $s>0.$
The authors of \cite{TMR} describe an example of a cover ideal of a graph failing the persistence property.
We show how to construct, starting from a hypergraph whose cover ideal fails the persistence property, a new hypergraph whose cover ideal fail such property.  
We use the notation introduced in Section \ref{s.main case}.

\begin{theorem}\label{t.persistence} 
	Let $H=(V,E)$ be a hypergraph where $V=X\cup \{y\}$ such that
	\begin{itemize}
		\item[1.] there exists only one edge $e_y\in E$ containing $y$; 
		\item[2.] $H$ has a shadow on $X$, say $H'=(X, E')\in \mathcal S(H).$ 
	\end{itemize}
	Suppose $J(H')$ fails the persistence property and let $\mathfrak p'\in \ass (J(H')^s)$ and $\mathfrak p'\notin \ass (J(H')^{s+1})$ for some $s>0$.  Set $\tilde H:=H_X$ the subhypergraph of $H$ on $X$.  
	If the following conditions hold,
	\begin{itemize}
		\item[3.] $\mathfrak p'\notin ass(J(\tilde H)^s)$; and
		\item[4.] $\mathfrak p'+(y)\notin ass(J(H)^{s+1}:y)$,
	\end{itemize}
	then $J(H)$ fails the persistence property.
\end{theorem}
\begin{proof}
	By hypothesis we have $\mathfrak p'\in \ass (J(H')^s)$ and $\mathfrak p'\notin ass(J(\tilde H)^s)$. So by Theorem 
	\ref{t.tildeH}, one gets that
	$$ \mathfrak p'+(y)\in \ass (J(H)^s).$$
	Moreover,  the hypothesis also ensures that
	$\mathfrak p'\in \ass (J(H')^{s+1})$ and $\mathfrak p'+(y)\notin \ass(J(H)^{s+1}:y)$. Hence, from Theorem 
	\ref{t.from H to H'}, we have $ \mathfrak p'+(y)\notin \ass (J(H)^{s+1}).$
\end{proof}

\begin{example}\label{e.persistence fails}
	In \cite{TMR}, Theorem 11 provides an example of a graph failing the persistence property.
	The graph, denoted by $H_4$, has the vertex set on $X:=\{x_1,\ldots, x_{12}\}$ and the edge set
	\begin{align*}
	E:=\{&\{x_1,x_2\}, \{x_1,x_5\}, \{x_1,x_9\}, \{x_1,x_{12}\},\{x_2,x_3\}, \{x_2,x_6\},\{x_2,x_{10}\}, \{x_3,x_4\},\\
	 &\{x_3,x_7\},\{x_3,x_{11}\},\{x_4,x_8\}, \{x_4,x_9\}, \{x_4,x_{12}\}, \{x_5,x_6\},\{x_5,x_8\},\{x_5,x_9\}, \\
	&\{x_6,x_7\},\{x_6,x_{10}\},\{x_7,x_8\},\{x_7,x_{11}\},\{x_8,x_{12}\},\{x_9,x_{10}\},\{x_{10},x_{11}\},\{x_{11},x_{12}\}\}.
	\end{align*}
	The persistence property fails since $\ass (J(H_4)^3)\not\subseteq \ass (J(H_4)^4)$.
	In particular, $\mathfrak p_X\in \ass (J(H_4)^3)\setminus \ass (J(H_4)^4)$.% and $\mathfrak p_X= J(H_4)^3: (x_1^2\cdots x_{12}^2)$.
	
	We consider now the hypergraph $H$ on  vertex set $V:=X\cup \{y\}$, constructed from $H_4$ by adding the variable $``y"$ only to the edge $\{x_1, x_2\}:$  $$H=(X\cup\{y\}, (E\setminus \left\{\{x_1,x_2\}\right\})\cup \left\{\{x_1,x_2, y\}\right\}).
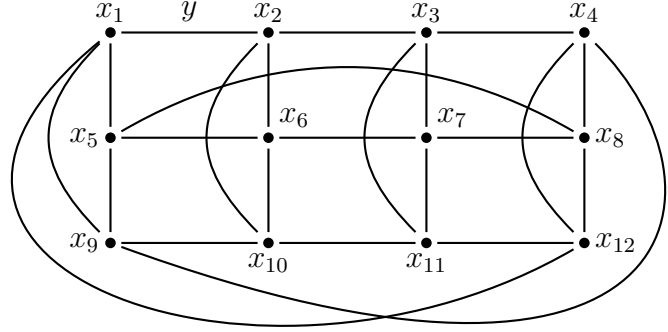
$$
	
	\begin{center}
		\begin{tikzpicture}[scale=0.7]
		\node (v1) at (0,0) {};
		\node (v2) at (3,0) {};
		\node (v3) at (6,0) {};
		\node (v4) at (9,0) {};
		\node (v5) at (0,-2) {};
		\node (v6) at (3,-2) {};
		\node (v7) at (6,-2) {};
		\node (v8) at (9,-2) {};
		\node (v9) at (0,-4) {};
		\node (v10) at (3,-4) {};
		\node (v11) at (6,-4) {};
		\node (v12) at (9,-4) {};
		\node (y) at (1.5,0)[above] {$y$};
		\draw [thick] (v1)--(v2);
		\draw [thick] (v3)--(v2);
		\draw [thick] (v3)--(v4);
		\draw [thick] (v1)--(v5);
		\draw [thick] (v5)--(v9);
		\draw [thick] (v5)--(v6);
		\draw [thick] (v6)--(v7);
		\draw [thick] (v7)--(v8);
		\draw [thick] (v4)--(v8);
		\draw [thick] (v8)--(v12);
		\draw [thick] (v9)--(v10);
		\draw [thick] (v10)--(v11);
		\draw [thick] (v11)--(v12);
		\draw [thick] (v2)--(v6);
		\draw [thick] (v10)--(v6);
		\draw [thick] (v3)--(v7);
		\draw [thick] (v7)--(v11);
		\draw [thick] (v1) .. controls (-1.5,-1.5) and (-1.5,-2.5) .. (v9);
		\draw [thick] (v2) .. controls (1.5,-1.5) and (1.5,-2.5) .. (v10);
		\draw [thick] (v3) .. controls (4.5,-1.5) and (4.5,-2.5) .. (v11);
		\draw [thick] (v4) .. controls (7.5,-1.5) and (7.5,-2.5) .. (v12);
		\draw [thick] (v5) .. controls (3,-0.25) and (6,-0.25) .. (v8);
		\draw [thick] (v9) .. controls (11,-8) and (12,-3) .. (v4);
		\draw [thick] (v1) .. controls (-5.5,-4.5) and (3,-7.5).. (v12);
		\fill (v1) circle (0.1) node [above] {$x_1$};
		\fill (v2) circle (0.1) node [above] {$x_2$};
		\fill (v3) circle (0.1) node [above] {$x_3$};
		\fill (v4) circle (0.1) node [above] {$x_4$};
		\fill (v5) circle (0.1) node [left] {$x_5$};
		\fill (v6) circle (0.1) node [above right] {$x_6$};
		\fill (v7) circle (0.1) node [above right] {$x_7$};
		\fill (v8) circle (0.1) node [right] {$x_8$};
		\fill (v9) circle (0.1) node [left] {$x_9$};
		\fill (v10) circle (0.1) node [below] {$x_{10}$};
		\fill (v11) circle (0.1) node [below] {$x_{11}$};
		\fill (v12) circle (0.1) node [right] {$x_{12}$};
		\end{tikzpicture}
		\vspace*{-40pt}
	\end{center}

	With this construction, $H_4$ is the shadow of $H$ on the set $X$. 
	Moreover, the subhypergraph of $H$ on $X$ is $$\tilde H=(X, E\setminus \left\{\{x_1,x_2\}\right\}).$$ 

	By Theorem  \ref{t.persistence}, $H$ fails the persistence property and
	$$\mathfrak p_V\in \ass (J(H)^3)\setminus \ass (J(H)^4).$$
	One can check, for instance using Macaulay2, that actually $\mathfrak p_X\notin ass(J(\tilde H)^3)$ and  $\mathfrak p_V\notin ass(J(H)^4:y).$
\end{example}

\begin{example}\label{e.persistence does not fail}
Take the hypergraph $H'$ on the vertex set $X=\{x_1,\ldots, x_{12}, x_{13}\}$ and the edge set
\begin{align*}
E=\{&\{x_1,x_2,x_{13}\}, \{x_1,x_5\}, \{x_1,x_9,x_{13}\}, \{x_1,x_{12},x_{13}\}, \{x_2,x_3,x_{13}\}, \{x_2,x_6,x_{13}\},\\
&\{x_2,x_{10},x_{13}\},\{x_3,x_4,x_{13}\}, \{x_3,x_7,x_{13}\},\{x_3,x_{11},x_{13}\},\{x_4,x_8,x_{13}\}, \{x_4,x_9,x_{13}\},\\
& \{x_4,x_{12},x_{13}\},\{x_5,x_6,x_{13}\}, \{x_5,x_8,x_{13}\},\{x_5,x_9,x_{13}\},\{x_6,x_7,x_{13}\},\{x_6,x_{10},x_{13}\},\\
&\{x_7,x_8,x_{13}\},\{x_7,x_{11},x_{13}\},\{x_8,x_{12},x_{13}\},\{x_9,x_{10},x_{13}\},\{x_{10},x_{11},x_{13}\},\{x_{11},x_{12},x_{13}\}\}.
\end{align*}
It was constructed from $H_4$, see example \ref{e.persistence fails}, by adding a new variable $``x_{13}"$ to all the edges but $\{x_1,x_5\}.$ 
Consider now the hypergraph $H$ on vertex set $V:=X\cup \{y\}$, constructed from $H'$ by adding the variable $``y"$ only to the edge $\{x_1, x_5\}$:  $$H=(X\cup\{y\}, (E\setminus \left\{\{x_1,x_5\}\right\})\cup \left\{\{x_1,x_5, y\}\right\}).$$
With this construction, $H'$ is the shadow of $H$ on the set $X$. 
A computation with Macaulay2 shows that $\mathfrak p_V\in J(H)^4$, but $\mathfrak p_X\notin J(H')^4$.
Indeed, we found two minimal monomials $m_1$, $m_2$ such that  $\mathfrak p_V= (J(H)^4:m_1)=(J(H)^4:m_2)$ that are
$$m_1:=x_1^2x_2^3x_3^3x_4^2x_5^2 x_6^3 x_7^2 x_8^3 x_9^3 x_{10}^2  x_{11}^3  x_{12}^3y,\ \ m_2:=x_1^2 x_2^2 x_3^2 x_4^2 x_5^2 x_6^2 x_7^2 x_8^2 x_9^2 x_{10}^2  x_{11}^2  x_{12}^2x_{13}y.$$
Both are divisible by $y.$

Then the conditions in the statement of Theorem \ref{t.persistence} are not satisfied. By Theorem \ref{t.from H to H'}, we get $\mathfrak p_V\in \ass (J(H)^4:y)$.
Using Macaulay2, one can check that even if the hypergraph $H'$ fails the persistence property,
and in particular $\mathfrak p_X\in \ass (J(H')^3)\setminus \ass (J(H')^4)$,
we have $\ass (J(H)^3)\subseteq \ass (J(H)^4).$
\end{example}

\end{document}